\documentclass [10pt]{amsart}
\usepackage{amsmath,amsfonts,amsthm,amssymb,pxfonts}
\usepackage{Tabbing}
\usepackage{lastpage}
\usepackage{mathtools}
\usepackage{esint}
\usepackage{mathrsfs}
\newtheorem{theorem}{Theorem}%[section]
\newtheorem{definition}[theorem]{Definition}

\newtheorem{lemma}[theorem]{Lemma}

\newtheorem{corollary}[theorem]{Corollary}
\newtheorem{proposition}[theorem]{Proposition}
\newtheorem{remark}[theorem]{Remark}

\def\R{{\mathbb R}}
\def\H{{\mathcal H}}
\def\eps{{\epsilon}}
\def\ga{{\gamma}}
\def\diam{{\mbox{diam}}}

\begin{document}

\title[Tangents of $\sigma$-finite curves]{Tangents of $\sigma$-finite curves and scaled oscillation}

\author{Marianna Cs\"ornyei and Bobby Wilson}
\address{Department of Mathematics, The University of Chicago, 5734 South University Avenue, Chicago, IL 60615, U.S.A.}
\thanks{}



\begin{abstract}
We show that every continuous simple curve with $\sigma$-finite length has a tangent at positively many points.  We also apply this result to functions with finite lower scaled oscillation; and study the validity of the results in higher dimension.
\end{abstract}

\maketitle

%%%%%%%%%%%%%%%%%%%%%%%%%%%%%%%%%%%%%%%%%%%%%%%%%%%%%%%%%%%%%%
%%%%%%%%%%%%%%%%%%%%%%%%%%%%%%%%%%%%%%%%%%%%%%%%%%%%%%%%%%%%%%
\section{Introduction}
Consider a continuous function $f :\Omega\subset \mathbb{R}^n \rightarrow \mathbb{R}$. We begin with the well-known notion of Lipschitz continuity and Rademacher's Theorem which states that if $f$ is Lipschitz continuous on the domain $\Omega$, then $f$ is differentiable almost everywhere on $\Omega$.   A generalization of Rademacher's Theorem, known as Stepanov's theorem, states that the conclusion of Rademacher's theorem holds under the assumption that the upper scaled oscillation, (see Def. \ref{osc}), is finite for almost every point in $\Omega$.  

A natural question to ask is whether or not the assumptions of Stepanov's theorem can be reduced to assuming the lower scaled oscillation is finite for almost all $x \in \Omega$ while preserving the conclusions. Balogh and Cs\"ornyei \cite{BaCs06} proved that one cannot simply replace upper scaled oscillation with lower scaled oscillation and acquire the same results.  Indeed, for $n=1$, they constructed a nowhere differentiable, continuous function $f: [0,1] \rightarrow \mathbb{R}$ such that $l_f(x)=0$ for a.e. $x \in [0,1]$. However, they were able to show that if the lower scaled oscillation is finite on all but countably many points of $[0,1]$, then $f$ is always differentiable on a set of positive measure, though it is not necessarily differentiable almost everywhere. 

The results of our paper were initiated by the observation (see Lemma \ref{prop1}) that the graph of such $f$ has $\sigma$-finite $\mathcal{H}^1$ measure and $f$ satisfies Luzin's $(N)$ condition.  This naturally leads to the question whether every curve of $\sigma$-finite length has a tangent at every point of a set of positive $\mathcal{H}^1$ measure.

%This question is one that was answered in some sense by Besicovitch, in a string a papers beginning with his paper on tangents to rectifiable curves \cite{Be44}.  Of course the interesting part of the $\sigma$-finite condition is that our curve could have infinite linear measure.  In 1956, Besicovitch published a paper \cite{Be56} arguing that every $\sigma$-finite simple curve has a tangent field defined on a set of infinite measure. In this definition, the tangents are defined in the measure sense, and in order to avoid the obstacle of defining such an object with respect to a set with infinite measure, at each point, the tangent is defined with respect to the finite measure subset to which it belongs.   Actually, it is also shown in \cite{Be56} that the Favard measure of the set at which the curve has no tangent is zero.

%In this paper, we indeed show that every curve has a tangent, in the pointwise sense, on a set of positive measure:

It is well-known that every curve of finite length has a
tangent at a.e. of its points. Therefore the answer to this question is positive when the curve can be decomposed into countably many sub-arcs, each of which has finite length. However, there are curves of $\sigma$-finite $\mathcal{H}^1$ measure whose every sub-arc is of infinite length.

Every set $E$ of $\sigma$-finite $\mathcal{H}^1$ measure can be decomposed into a rectifiable and a purely unrectifiable part. The tangent of a rectifiable set is well-understood, for a detailed discussion see e.g. Chapter 15 in \cite{Ma}. The purely unrectifiable part has zero projection in a.e. direction, therefore, if $E$ is a curve, then its rectifiable part cannot be negligible, so the tangent of $E$ is defined at positively many points. Besicovitch discusses the special case when the set is a simple curve in \cite{Be56} and \cite{Be57}.
The definition of tangent was later generalized to an arbitrary Lebesgue null set in \cite{ACP}.

Both the standard definition of the tangents of rectifiable sets and the definition in \cite{ACP} for Lebesgue null sets are given in the ``almost everywhere" sense: if we change the tangent field on a sufficiently small set (namely, on a purely unrectifiable set), it is still a tangent field. Therefore these definitions cannot say what the tangent of the set is at a given point. In \cite{Be57}, Besicovitch showed that there exists a continuous function defined on $[0,1]$ with a $\sigma$-finite graph and for which there exists a set $E \subset [0,1]$, with $\mathcal{L}^1(E)=\frac{1}{2}$, such that $f$ is not differentiable on $E$. He also  asserted with his demonstration of this example that ``...it is hardly possible to find a satisfactory definition of a tangent to a curve of $\sigma$-finite length at a point considered individually."

However, in this paper we show that every $\sigma$-finite curve has a tangent, in the pointwise sense, on a set of positive measure. By ``tangent" we mean the classical notion: the limit direction from which the curve approaches the point (see Section \ref{prelims}). We prove the following theorem:

	\begin{theorem} \label{main}
	Let $C$ be a closed, continuous simple curve in $\R^n$. Assume that $C$ has  $\sigma$-finite $\mathcal{H}^1$ measure.
Then $C$ has a tangent at positively many of its points.
%every point of a set of positive $\mathcal{H}^1$ measure.
	\end{theorem}

We also show how this theorem can be used to give a very quick proof of the result of \cite{BaCs06} in dimension $n=1$ mentioned above, and we also extend it to higher dimension. 

In \cite{BaCs06}, it was proved that a function $f: [0,1]^n \rightarrow \mathbb{R}$ is differentiable almost everywhere, provided that $l_f(x)<\infty$ except at the points of a set $E$ with $\sigma$-finite ${\mathcal H}^{n-1}$ measure, and $l_f\in L^p$ where $p>n$. Hanson showed in \cite{Han} that one cannot weaken the hypotheses that $E$ has $\sigma$-finite $(n-1)$-dimensional measure: he gave a sharp estimate for the gauge functions $g$ for which the theorem remains true with ${\mathcal H}^{n-1}$ replaced with ${\mathcal H}^g$. In particular, he showed that $E$ can have Hausdorff dimension at most $n-1$.

In view of the positive results for $n=1$ mentioned at the beginning of this introduction, it is natural to conjecture that, though a function $f: [0,1]^n \rightarrow \mathbb{R}$ may not be differentiable almost everywhere, it is still differentiable on a set of positive measure, provided that $l_f(x)<\infty$ outside a set of  $\sigma$-finite ${\mathcal H}^{n-1}$ measure. Here we do not assume that $l_f\in L^p$.

Surprisingly, we show that this conjecture is false for $n>1$. Theorem \ref{main} holds true only in dimension one, higher dimensional $\sigma$-finite surfaces do not necessarily have tangents at any points. Through this observation we can construct non-differentiable functions for {\it any gauge function $g$}. We prove the following theorem:

\begin{theorem} \label{main2}
	For any $n\ge 1$ and for any non-zero gauge function $g$ there exists an almost everywhere non-differentiable continuous function $f: [0,1]^n \rightarrow \mathbb{R}$ with $l_f(x)<\infty$ at all points in $[0,1]^n$ except on a set $E$ with $\mathcal{H}^g(E)=0$.
	\end{theorem}

However, for $\H^0$, we obtain a positive result in any dimension. This answers a question of Hanson.

\begin{theorem} \label{fdiffalld}
	Let $f:[0,1]^n\rightarrow \mathbb{R}$ be a continuous function. Assume $l_f(x)<\infty$ for all but at most countably many $x \in [0,1]^n$.  Then $f$ is differentiable at the points of a set of positive measure.
	\end{theorem}

Geometrically, the main difference between an exceptional set of $\sigma$-finite $\mathcal{H}^{n-1}$ (or, $\mathcal{H}^g$) measure and a countable exceptional set is that the former one may be mapped onto a large set by a continuous function, while, no matter what, the image of a countable set is always countable. Although, the validity of Theorem \ref{fdiffalld} was an open problem, even when assuming the exceptional set is empty.

%(However, the statement is non-trivial, and was an open problem, even when the exceptional set is empty.)

%\bigskip
%Our paper is organized as follows.
%
%In Section \ref{prelims}, we state some preliminary definitions and results from Saks \cite{Sa37}.  In Section \ref{sigma}, we establish the differentiability of $\sigma$-finite continuous curves and, in the next section, this is applied to functions $f:[0,1] \rightarrow \mathbb{R}$ with finite lower scaled oscillation.  The following section, Section \ref{higher}, we apply Section \ref{scaled}, to functions $f:[0,1] \rightarrow \mathbb{R}^n$ with finite lowers scaled oscillation, and construct a $\sigma$-finite $n-1$-dimensional surface in $\mathbb{R}^n$ that disproves the extension of Theorem \ref{main} to higher dimensional surfaces.  In Section \ref{proofs}, we prove Theorem \ref{main2}, and in the final section we prove Theorem \ref{fdiffalld}.

\section{Preliminaries}\label{prelims}
To start with, we fix some notations. As usual, we let $\mathcal{L}^n$ denote the $n$-dimensional Lebesgue measure and $\mathcal{H}^s$ denotes the $s$-dimensional Hausdorff measure for $n \in \mathbb{N}$ and $s \in [0,\infty)$. We will also use Hausdorff measures $\H^g$ where $g$ is a gauge function. Recall that a function $g:[0,\infty) \rightarrow [0,\infty)$ is called a gauge function, if it is continuous, monotone increasing, $g(0)=0$ and $g(t)>0$ for $t>0$. The Hausdorff measure $\mathcal{H}^g$ is defined as the limit, $\lim_{\delta\to 0} \mathcal{H}_\delta^g$, where for a given set $E$, 
\begin{align*}
\mathcal{H}_\delta^g(E)=\inf\left\{\sum_i g(\mbox{diam}(E_i))~\big|~E\subset\bigcup_i E_i,\  \mbox{diam}(E_i)<\delta\right\}.
 \end{align*}  

We denote the open ball around $x$ of radius $r$ by $B(x,r)$, and, for any direction $\theta$, the orthogonal projection onto $\theta^{\perp}$ by $\mbox{proj}_{\theta}$.  

With respect to inequalities, we will commonly use $a \lesssim b$ to denote instances when there exists a constant $C \in (0,\infty)$ such that $a \leq Cb$ and $a \ll b$ in situations where $a<Cb$ for some positive constant $C$ much smaller than 1. 

We define:
	\begin{definition}[Upper/Lower Scaled Oscillation] \label{osc}
	Let $(X,d)$ be a metric space, $f : X \rightarrow \mathbb{R}^n$ be continuous.  For $x \in X$ define the upper and lower scaled oscillation functions of $f$ as:
	\begin{align*}
	 L_f(x):= \limsup_{r\rightarrow 0} \frac{\sup_{d(x,y)\leq r} \|f(y)-f(x)\|}{r}\\
	 l_f(x):= \liminf_{r\rightarrow 0} \frac{\sup_{d(x,y)\leq r} \|f(y)-f(x)\|}{r}
	\end{align*}
	\end{definition}

Next, we state some definitions and a theorem that appear in \cite{Sa37}. The half-line in $\mathbb{R}^n$ issuing from a point $x$ and containing a point $y \neq x$ will be denoted by $\vec{xy}$. Given $E \subset \mathbb{R}^n$, a half-line $L$ issuing from a point $x \in E$ is called an {\it intermediate half-tangent} of $E$ at $x$, if there exists a sequence $\{y_n\}$ of points of $E$ distinct from $x$, converging to $x$ and such that the sequence of half-lines $\{\vec{xy_n}\}$ converges to $L$.  
%Now, we can state the definition of the contingent of a set at a point.

\begin{definition}[Contingent]
The set of all intermediate half-tangents of a set $E$ at a point $x$ is termed the  contingent of $E$ at $x$ and denoted by $\mbox{contg}_E x$.
\end{definition}

When $E\subset\R^n$ is homeomorphic to $[0,1]^k$, and the contingent of $E$, at a point $x$, is a $k$-dimensional plane, then we say that it is the {\it tangent of $E$ at $x$}. Indeed, geometrically, this means that locally around $x$ the set $E$ can be very well approximated by this $k$-plane: in a small enough neighborhood of $x$, all the half-lines $\vec{xy}$ are within a very small angle of the plane. In addition, close to any direction in the plane, there is an $\vec{xy}$. 

In the special case when $E$ is the graph of a continuous function $f$, the function $f$ is differentiable at $x$ if and only if its graph has a tangent at the point $(x,f(x))$ and this tangent is not a vertical plane (i.e. it does not contain the vertical line). Also note that, for any function $f$, at least one of the two vertical half-lines belongs to the contingent of the graph of the function at $(x,f(x))$ if and only if $L_f(x)=\infty$. On the other hand, $l_f(x)<\infty$ implies that the contingent contains at least one non-vertical half-line.

%.  This definition implies the exists of partial derivatives, but also coincides with the classical definition of the derivative because there are no restrictions on the the types of sequences, $\{a_n\}$, that we consider for our secant half-lines.  Specifically, nondifferentiability implies that for any plane at $(y,f(y))$, $P_L=\{z \in \mathbb{R}^m ~|~ (y,f(y))+(z,L\cdot z)\}$ there is a sequence, $\{\xi_n\} \subset \mbox{graph of } f$, such that $\xi_n \rightarrow x=(y,f(y))$ and the set of angles between the $\xi_n$ and $P_L$ is bounded away from 0.  However,  $\lim _{n \rightarrow \infty}  \vec{x\xi_n}$ is in the contingent (passing to a subsequence if necessary), so the contingent of the graph of $f$ at $x$ will not be contained in a tangent plane if $f$ is not differentiable.

%
%
%Our goal is to apply this theorem to curves that are not graphs of functions.  In order to accomplish this, we instead use a key component to the theorem.  The DYS theorem is partially derived from the 

We will use the following lemma from Saks \cite{Sa37}:

	\begin{lemma}[Theorem 3.6, Chapter IX in Saks \cite{Sa37}]\label{3.6}
	Given a planar set $R$, let $P$ be a subset of $R$ at no point of which the contingent of $R$ is the whole plane. Then
	\begin{enumerate}
		\item[i)] the set $P$ is a countable union of sets of finite length and 
		\item[ii)] at every point of $P$, except at those of a set of length zero, either the set $R$ has a tangent or else the contingent of $R$ is a half-plane.
	\end{enumerate} 
	\end{lemma}

The analogous result can also be proved for sets in $\mathbb{R}^n$ for any $n$ by the same argument as the planar proof in \cite{Sa37} (see \cite{Rog35}). At each point except at those of a set of zero ($n-1$)-dimensional measure, the contingent is either the whole space, or a halfspace, or an ($n-1$)-plane. In this paper, in dimensions $n\ge 3$, we will use only the following corollary, that can also be thought of as a variant of Stepanov's theorem mentioned in the introduction:

\begin{corollary}\label{lf} Let $f: \mathbb{R}^n \rightarrow \mathbb{R}$ be a continuous function. Then $L_f(x)=\infty$ at almost every point $x$ where $f$ is non-differentiable. 
\end{corollary}

%For the sake of completeness, we also remark that Lemma \ref{3.6} can be used to prove Denjoy-Young-Saks theorem that describes the behaviour of the Dini-derivatives of a continuous function $f: \mathbb{R} \rightarrow \mathbb{R}$, see e.g. in ... However, in higher dimension, one encounters difficulties, the higher dimensional analogue of the Denjoy-Young-Saks theorem for partial Dini derivatives is false (see ...). This is because the contingent of the graph of a function does not really see the partial Dini derivatives: it also takes into account sequences whose directions converge to the coordinate direction.

\section{$\sigma$-finite curves and differentiability}\label{sigma}

In what follows, we will use the notation $F^{\theta}_x$ for the 
($n-1$)-dimensional plane through the point $x\in{\mathbb R}^n$ of normal direction $\theta$.

First, let $C$ be a continuous simple curve in the plane. We show that we have control of the contingent for points that are in some sense isolated with respect to the curve:

	\begin{lemma}\label{dlemma2}
	Let $C$ be a continuous, closed simple curve in the plane and suppose that on a set $D\subset C$, there exist three directions $\theta_1$, $\theta_2, \theta_3$ such that for every $z\in D$, $z$ is isolated in $C \cap F^{\theta_i}_{z}$ for each $i \in \{1,2,3\}$.  Then at almost all $z \in D$, $C$ has a tangent line.
	\end{lemma}
\begin{proof}
For a fixed $z \in D$, there exists a $\delta$-ball around $z$ such that $F^{\theta_i}_{z}\cap C\cap B(z,\delta)=\{z\}$ for each $i$.  The $F^{\theta_i}_{z}$ decompose the $\delta$-ball around $z$ into 6 regions. The curve is simple and continuous, so the curve can only lie in at most 2 of the 6 regions around $z$. Therefore the contingent of $C$ cannot be the whole plane or a halfplane. By Lemma \ref{3.6}, at all points of $D$ except at those of a set of length zero, $C$ has a tangent.
\end{proof}

\begin{remark}\label{remark} We observe that it suffices to assume that the curve is injective at almost every point of $D$ to prove Lemma \ref{dlemma2}.  Any self-intersections away from points in $D$ will not affect the argument. 
%We will use this remark in Section \ref{higher}.
\end{remark}

In the planar case we prove Theorem \ref{main} by showing that if $C$ has 
$\sigma$-finite $\mathcal{H}^1$ measure then there exists a set $D$ of positive   $\mathcal{H}^1$ measure for which Lemma ~\ref{dlemma2} can be applied. This is a corollary of a theorem from Falconer \cite{Fa85}.

	\begin{lemma}[Falconer \cite{Fa85}, Corollary to Theorem 5.8] \label{proj}
	Consider $E \subset \mathbb{R}^d$ and let $A$ be a subset of a primary axis. Let $E_x=F^{\theta}_x\cap E$, where $\theta$ is the direction of the primary axis. Suppose that if $x \in A$, then $\mathcal{H}^0(E_x)>c$, for some constant $c$.  Then
	\begin{align*}
	\mathcal{H}^s(E) \geq bc\mathcal{H}^s(A)
	\end{align*}
	where $b$ depends on $s$ and $d$. 
	\end{lemma}

\begin{proof}[Proof of Theorem \ref{main} in the plane]
Let $\theta \in (0,\pi)$ be some angle and $L_{\theta}$ be the line with angle $\theta$ through the origin.  We first note that 
\begin{align*}
\mathcal{H}^1\left[\mbox{proj}_{\theta}(C) \right]>0
\end{align*}
for all but at most one $\theta$.  Therefore, without loss of generality we assume $\mathcal{H}^1\left[\mbox{proj}_{\theta}(C) \right]>0$. Consider the decomposition of $C$, $\{C_M\}$, given by the $\sigma$-finite condition, and denote $E_M:=\bigcup_1^M C_i$.  Since 
\begin{align*}
\mbox{proj}_{\theta}(E_M)\xrightarrow{M\rightarrow \infty} \mbox{proj}_{\theta}(C),
\end{align*}
  there exists an $M_0$ such that $\mathcal{H}^1\left[\mbox{proj}_{\theta}(E_M)\right]>0$ for $M\geq M_0$.  Now
\begin{align*}
0<\mathcal{H}^1\left[\mbox{proj}_{\theta}(E_{M_0})\right]\leq\mathcal{H}^1(E_{M_0})<\infty.
\end{align*}
By Lemma \ref{proj}, 
\begin{align*}
\mathcal{H}^0\left[E_{M_0}\cap F^{\theta}_z\right]<\infty
\end{align*}
for almost all $z \in \mbox{proj}_{\theta}(E_{M_0})$.  Furthermore, for almost all $z \in \mbox{proj}_{\theta}(E_{M_0})$, 
\begin{align*}
\mathcal{H}^0\left[E_M\cap F^{\theta}_z\right]<\infty
\end{align*}
for $M>M_0$.  Thus, $C\cap F^{\theta}_z$ is at most countable for almost all $z \in \mbox{proj}_{\theta}(E_{M_0})$.  

Every closed countable set of points must contain an isolated point. Let $X_{\theta}$ be the set of points in $C$ that are isolated in $C\cap F^{\theta}_z$ for some $z$. Then $X_\theta$ contains at least one point of $C\cap F^{\theta}_z$ for almost all $z \in \mbox{proj}_{\theta}(E_{M_0})$.

In order to show that $X_{\theta}$ is a Borel set, we can show that its complement is an $F_{\sigma \delta}$ set.  Consider, for each $n,m \in \mathbb{N}$, $n>m$, the set
\begin{align*}
F_{n,m}:=\left\{x \in C ~\big|~ \mbox{There exists a } y \in C\cap F^{\theta}_x  \mbox{ such that } \|x-y\| \in \left[\frac{1}{n}, \frac{1}{m}\right] \right\}.
\end{align*}
Since $C$ is closed $F_{n,m}$ is closed. We take the union over $n$ and then the intersection over $m$ and we obtain the complement to $X_{\theta}$.

%Since $\mathcal{H}^1\left[\mbox{proj}_{\theta}(E_{M_0})\right]>0$ and $X_\theta$ contains at least one point of $C\cap F^{\theta}_z$ for almost all $z \in \mbox{proj}_{\theta}(E_{M_0})$, 
By Lemma \ref{proj}, $\mathcal{H}^1(X_{\theta})>0$.
% where $X_{\theta}:=\{y_z ~|~ z \in Z_{\theta}\}$.
% and by Proposition \ref{prop1}, if we let $X_{\theta}:=\{ x \in I~|~ (x,f(x))=y_z \mbox{ for some }y_z\}$,
%\begin{align*}
%\mathcal{H}^1(X_{\theta})>0.
%\end{align*}
This holds true for uncountably many $\theta$.  Therefore, there exist 3 angles 
$\theta_1$, $\theta_2$, $\theta_3$ such that 
\begin{align*}
\mathcal{H}^1\left(\bigcap_{i} X_{\theta_i}\right)>0
\end{align*}
and we can apply Lemma \ref{dlemma2} with $D=\bigcap_{i} X_{\theta_i}$ to obtain positively many points at which $C$ has a tangent.
\end{proof}

\begin{proof}[Proof of Theorem \ref{main} in higher dimension]
We show that our arguments can be extended to simple curves in $\mathbb{R}^n$. 
%Instead of angles in the case of the plane, we consider directions $\theta \in S^{n-1}$. 
Let us first consider the case $n=3$.

For any direction $\theta\in S^2$, if $\mbox{proj}_\theta C$ has a tangent $\tau$ at $\mbox{proj}_\theta x$, then $\mbox{contg}_Cx$ must lie inside the 2-dimensional plane $\mbox{proj}_\theta^{-1} \tau$.
For almost all directions, almost all 2 dimensional plane sections will intersect the curve at isolated points (with respect to the section) for positively many points of the curve. By Lemma \ref{dlemma2} and Remark \ref{remark}, it is sufficient to find 6 directions $\eta_1$, $\eta_2$, $\eta_3$ and $\eta_4$, $\eta_5$, $\eta_6$: two triples belonging to two different planes $\theta_1^\perp$, $\theta_2^\perp$, such that each point $x$ in a subset $D\subset C$ is isolated in each of the six 2-dimensional sections $F_x^{\eta_j}\cap C$, and for which $D$ has positive measure. Then for each $i=1,2$ and for a sufficiently small $\delta$, we can find tangent $\tau_i$ of $\mbox{proj}_{\theta_i}(C\cap B(x,\delta))$ at $\mbox{proj}_{\theta_i}x$ at positively many $x\in D$, and the tangent of $C$ at $x$ will be the line $\bigcap_i\mbox{proj}_{\theta_i}^{-1}\tau_i$.

The higher dimensions follow by taking more planes.
\end{proof}

%
%

%
%
%
%
%
%%%%%%%
%
%\section{Scaled oscillation, Luzin's Condition and $\sigma$-Finite Curves}\label{scaled}

%We now include the application to the result from \cite{BaCs06} on lower scaled oscillation.
%

Now, let $f:\R^n\rightarrow \mathbb{R}^m$ be a continuous mapping, and let $G_f(E)$ denote the graph of $f$ on a set $E \subset \mathbb{R}^n$. The following lemma shows that the finite lower scaled oscillation implies a certain $\sigma$-finite condition and a Luzin type condition. Actually, we will use this lemma only in the special case when $m=s=1$ and $E$ is a simple curve of finite length. However, the validity of this lemma in higher dimensions is the main motivation behind our construction of a surface $S$ in Section \ref{higher} that we will use in the proof of Theorem \ref{main2}. So let us state our lemma in its full  generality:

	\begin{lemma} \label{prop1}
Let $s>0$ be arbitrary, let $E\subset \mathbb{R}^n$ be a Borel set, and let $f:E\rightarrow \mathbb{R}^m$ be a continuous mapping satisfying $l_f(x)<\infty$ and $\liminf_{r\to 0}\mathcal{H}^s(E\cap B(x,r))/r^s>0$ at all but countably many points $x\in E$.
Then 
	\begin{enumerate}
		\item $G_f(E)$ has $\sigma$-finite $\mathcal{H}^s$-measure.
		\item $\mathcal{H}^{s}[G_f(H)]=0$ for every set $H \subset E$ with $\mathcal{H}^{s}(H)=0$.
	\end{enumerate}
	\end{lemma}

Note that in the special case when $E$ is a continuous curve and $s=1$, the condition $\liminf_{r\to 0}\mathcal{H}^1(E\cap B(x,r))/r>0$ is automatically satisfied. However, this is not true for $k$-dimensional surfaces and $\H^k$.
%
%We also define Luzin's condition:
%
%	\begin{definition}[The Luzin Condition $(N)$, $(N^*)$]
%A function $f:\mathbb{R}^k\rightarrow \mathbb{R}^{n}$ is said to fulfill the condition $(N)$ on a set $E$, if for every set $H \subset E$ of Lebesgue measure zero, the image of $H$ through $f$ is Lebesgue measure zero.
%
%	If  $\mathcal{H}^{k}[G_f(H)]=0$ for every set $H \subset E$ of Lebesgue measure zero, then $f$ is said to fulfill the condition $(N^*)$.
%	\end{definition}
%
%We are now prepared to reduce the lower scaled oscillation condition to the $\sigma$-finite and Luzin conditions.   It is proved in the following that the lower scaled oscillation condition everywhere on a set is enough to show that there is a decomposition of the domain of $f$ into sets for which the size of the graph of $f$ on each set is controlled by the size of the corresponding set in the domain. 
%
%	\begin{proposition} \label{prop1}
%	Let $f:\mathbb{R}^k\rightarrow \mathbb{R}^n$ be a function satisfying $l_f(x)<\infty$ for all $x \in E$ for some $E\subset \mathbb{R}^k$ with $\mathcal{L}^k(E)<\infty$. Then 
%	\begin{enumerate}
%		\item $G_f(E)$ is  $\mathcal{H}^k$ $\sigma$-finite.
%		\item $f$ fulfills condition $(N^*)$ on $E$.
%	\end{enumerate}
%	\end{proposition}
%
\begin{proof}
The image of the countable exceptional set is countable, therefore it is enough to prove that the graph above the non-exceptional points have $\sigma$-finite $\mathcal{H}^s$ measure and that this part of the graph satisfies Luzin condition (2).

For given positive numbers $M$ and $t$, denote $$F=E_{M,t}:=\{ x \in E\,|\, l_f(x) < M,\ \liminf_{r\to 0}{\mathcal{H}^s(E\cap B(x,r))/r^s}> t\}.$$ For part (1), it is sufficient to prove that $\mathcal{H}^s(G_f(F))<\infty$ for any $M$ and $t$. 

The hypothesis implies that for an arbitrary small $\delta$ there exists a covering of $F$ by balls of diameter less than $\delta$, such that for each ball $B=B(x,r)$, $\|f(x)-f(y)\|\leq Mr$ for all $y\in B$, and
$\mathcal{H}^s(E\cap B)\ge tr^s$.
%. We can also assume $\mathcal{L}^k(\cup B_r) \leq \mathcal{L}^k(H_M)+\delta$.  

Now let $\{ B_i \}$ (with centers $x_i$ and radii $r_i$) be a subcover given by the Besicovitch Covering Theorem. We note that
$\|f(x_i)-f(y)\| \leq Mr_i$
for every $y\in B_i$. Then, with $c_M=\sqrt{M^2+1}$, the balls around the points $(x_i,f(x_i))$ of diameter $c_M \mbox{diam}(B_i)<c_M\delta$ cover $G_f(F)$, therefore 
\begin{align*}
\mathcal{H}^s_{c_M\delta}(G_f(F)) &\leq \sum_i  (2 c_M\mbox{diam}(B_i))^s \leq c_{M,s,t}\sum_i\mathcal{H}^s(E\cap B_i)\le\\
&\le c_{n,M,s,t}\mathcal{H}^s(E)<\infty.
\end{align*}

For part (2), consider $H_{M,t}:=H\cap E_{M,t}$, and $\eps>0$ arbitrary, and cover $H_{M,t}$ by a relatively open set $U\subset E$ with $\mathcal{H}^s(U)<\eps/c_{n,M,s,t}$. The hypothesis of our lemma remains true with $E$ replaced by $U$, therefore our argument above shows that $\mathcal{H}^s(G_f(H_{M,t}))\le \mathcal{H}^s(G_f(U_{M,t}))\le c_{n,M,s,t}\mathcal{H}^s(U)<\eps$.
This holds for any $\eps$, therefore $\mathcal{H}^s(G_f(H_{M,t}))=0$.
\end{proof}

Theorem \ref{main} combined with Lemma \ref{prop1} gives an alternative 
proof to the positive result in \cite{BaCs06} for $n=1$ mentioned in the beginning of the introduction. Moreover:

	\begin{corollary} \label{fdiff}
	Let $\gamma\subset[0,1]^n$ be a simple curve with finite length, and let $f:\gamma\rightarrow \mathbb{R}$ be a continuous function. Assume $l_f(x)<\infty$ for all but at most countably many $x$ in $\gamma$. Then $f$ is differentiable at positively many points of $\gamma$.
%	\begin{align*}
%	\mathcal{L}^1\{ x \in (0,1) ~|~  f'(x) \mbox{ exists}\}>0.
%	\end{align*}
	\end{corollary}

Indeed, applying part (1) of Lemma \ref{prop1} with $E=\gamma$ and $m=s=1$, we can see that $G_f(E)$ is a simple curve in $\R^{n+1}$ of $\sigma$-finite length. By Theorem \ref{main}, it has a tangent at a subset $D$ with $\mathcal{H}^1(D)>0$. By part (2), this set $D$ lies above a positive subset of $\gamma$. In order to prove that $f$ is differentiable at these points, the only thing we need to check is that the tangent line is not vertical. This follows immediately from the property $l_f(x)<\infty$.

\medskip
We now show how Corollary \ref{fdiff} can be used to prove Theorem \ref{fdiffalld}. The key step in our proof is the following lemma:

\begin{lemma}\label{above}
Let $f:[0,1]^n\rightarrow \mathbb{R}$ be a continuous function, and assume $L_f(x)=\infty$ at almost every $x\in[0,1]^n$. Then there exists a simple curve $\gamma\subset[0,1]^n$ of finite length, such that $L_g(x)=\infty$ for $g:=f|_\gamma$ and for $\H^1$-a.e. $x$ in $\gamma$.
\end{lemma}

Assume that Theorem \ref{fdiffalld} is false for some function $f$. Then, by Corollary \ref{lf}, it satisfies the hypotheses of Lemma \ref{above}. This gives us a curve $\gamma$ with the property that for almost all $x$ in $\gamma$, the contingent of the graph of $g$ at $(x,g(x))$ contains at least one vertical direction. Clearly, $l_g\le l_f$ at every point. Therefore, applying Corollary \ref{fdiff} for $\gamma$ and $g$, at positively many points, the contingent is a non-vertical line. This contradiction proves Theorem \ref{fdiffalld}. It remains to prove Lemma \ref{above}:

\begin{proof}
By Fubini's theorem, there a straight line segment $I\subset[0,1]^n$ such that
$L_f=\infty$ at almost every point of $I$.
We denote $I$ as $\gamma_0$ and begin our construction of $\gamma$. Fix a small $\eps_0>0$. At each point $x\in I$ where $L_f(x)=\infty$, for an arbitrary small $r$, there is a point
$y\in [0,1]^n$ with $\|x-y\|< r\eps_0$ and $|f(x)-f(y)|> r/\eps_0$. The balls $B(x,r)$ form a Vitali covering, therefore we can choose countably many disjoint balls $B_i=B(x_i,r_i)$ that cover almost every point $x$ in $I$ with $L_f(x)=\infty$.
For each $i$, replace the part of $\gamma_0$ that is in $B_i$, i.e. the line segment $[a_i,b_i]:=I\cap B_i$, with a polygon from $a_i$ to $b_i$ that has $x$ and $y$ as its vertices (and possibly it also has some other vertices). We define this new curve as $\gamma_1$. We can choose the polygons such that the length of $\ga_1$ is less than $(1+\eps_0)|I|$. 

Suppose that we have constructed $\gamma_j$.  In order to construct $\gamma_{j+1}$, we choose a small $\eps_j$ and perform the same transformation that we applied to $\gamma_0$ to each segment of $\gamma_j$.  However, one obstacle to performing the transformation may be that $L_f(x)<\infty$ on a segment of $\gamma_j$.  To avoid this situation, we can first replace each segment by a polygon, arbitrarily close to the original one, on which $L_f=\infty$ at $\mathcal{H}^1$ almost every point. Now we can transform each segment of $\gamma_j$. It is clear from the geometry that, by choosing the parameters $\eps_j$ and the balls in the Vitali coverings small enough, in the limit, we obtain a simple curve $\ga$ of length at most $\prod_j(1+\eps_j)\,|I|<\infty$. 

The curve $\gamma$ contains countably many exceptional sets of $\H^1$ measure zero (namely, those points that are not covered when we apply Vitali's covering theorem), but all other points $x_0\in\gamma$ are covered by a ball $B(x,r)$ for some $x\in\gamma$ and an arbitrary small $r$, and there is a $y\in B(x,r)\cap\gamma$ and an arbitrary small $\eps$, such that $|f(x)-f(y)|>r/\eps$. Then of course
$\mbox{max}(|f(x_0)-f(x)|/r,|f(x_0)-f(y)|/r)>1/2\eps$, so indeed $L_g(x_0)=\infty$ for $g=f|_\gamma$.
\end{proof}

\section{Higher dimensional surfaces and non-differentiability}\label{higher}

%First  This also shows that Corollary \ref{fdiff} holds for vector-valued mappings:
%
%	\begin{corollary} \label{fdiff2}
%	Let $f:[0,1]\rightarrow \mathbb{R}^n$ be a continuous function. Assume $l_f(x)<\infty$ for all but at most countably many $x \in [0,1]$.  Then
%	\begin{align*}
%	\mathcal{L}^1\{ x \in (0,1) ~|~  f'(x) \mbox{ exists}\}>0.
%	\end{align*}
%	\end{corollary}

We devote the rest of this paper to the construction of an
$n$-dimensional surface in ${\mathbb R}^{n+1}$ of finite ${\mathcal H}^{n}$ measure, for any $n>1$, that does {\it not} have a tangent at any of its points (and also have some additional properties). And we show how this construction can be used to prove Theorem \ref{main2}. 

\begin{proposition}\label{propp}

\begin{itemize}
\item[(1)] For any $n>1$ there exists a set $S \subset \mathbb{R}^{n+1}$ homeomorphic to $[0,1]^n$ that has finite $\mathcal{H}^n$ measure, but it does not have a tangent plane at any of its points. 
\item[(2)]Furthermore, for any $n\ge 1$ and for any gauge function $g$, there exists a set $S \subset \mathbb{R}^{n+1}$ and a homeomorphism $h: [0,1]^n \rightarrow S$ that satisfies $l_h(x)<\infty$ at every $x\in[0,1]^n$ except at the points of a set of zero $\mathcal{H}^{g}$ measure, and $h$ is not differentiable anywhere.
\end{itemize}
\end{proposition}

Proposition \ref{propp} implies Theorem \ref{main2}. Indeed, consider the sets
\begin{align*}
D_{v}:= \{x \in [0,1]^n ~|~ v \cdot h \mbox{ is differentiable at } x\} \end{align*}
where $v\in\mathbb{R}^{n+1}$ with $\|v\|=1$, $h$ is the nowhere differentiable homeomorphism in Proposition \ref{propp}, and $v \cdot h$ is the usual scalar product.
If every $D_{v}$ has positive $\mathcal{L}^n$ measure, then there are $n+1$ linearly independent directions $v_i$ such that $D:=\bigcap_i D_{v_i}\neq\emptyset$, and then $h$ would be differentiable at the points of $D$. Therefore,  there exists $v$ such that $f:=v \cdot h: \mathbb{R}^n \rightarrow \mathbb{R}$ is not differentiable at a.e. $x$. And since the scalar product is Lipschitz with Lipschitz constant 1, $l_f(x)\le l_h(x)$ at every $x$.

\medskip
Now we turn to the proof of Proposition \ref{propp}.

\subsection{Construction of $S$}\label{S}
We will construct our surface for $n=2$ for simplicity. An analogous construction will work in any dimension. 

Let $S_0$ be a (closed) planar unit square in $\mathbb{R}^3$. For any $j\ge 0$, we will define a surface $S_j\subset\R^3$ inductively, and then we will take the limit of these surfaces. Each surface $S_j$ will be a union of countably many ``edges'' (straight line segments) and countably many open ``faces'' (2-dimensional planar polygonal regions whose boundary line segments are edges of $S_j$). Naturally, the edges of $S_0$ are the four edges of the square, and it has one face, the interior of the square.

Suppose that we already defined $S_j$ for some $j\ge 0$. First we decompose the faces of $S_j$ into countably many non-overlapping closed squares such that for any two squares $Q,Q'$, if the intersection $3Q\cap 3Q'$ is non-empty, then they are of comparable sizes. We will use the notation $\ell(Q)$ for the sidelength of a square $Q$. We also assume that, for each of our squares $Q$, $\ell(Q)\ll\mbox{dist}(Q,S_j\setminus F)$, where $F$ is the face of $S_j$ to which $Q$ belongs. In particular, $Q$ has distance $\gg \ell(Q)$ from the edges of $F$.

Next, we modify the surface $S_j$ in the interior of the squares $Q$ to obtain $S_{j+1}$. Consider a square $Q$ of our construction, and let $Q^*$ be the square centered at the center of $Q$ with sidelength $\ell(Q^*) \ll \ell(Q)$. We place a square pyramid $P$ of angle $\alpha(Q)\ll 1$ onto $Q$, and a square column $C$ of height $\ell(Q)$ onto $Q^*$, respectively, as in Figure 1. Then we replace $\mbox{int}\,Q$ by $\partial(C\cup P)\setminus Q$. We do this inside each of our squares $Q$, with sufficiently small $\ell(Q^*), \alpha(Q)$ that we will specify later. The resulting surface is $S_{j+1}$.

$$
\begin{array}{c} \mbox{Figure 1}\\
\includegraphics[scale=0.4]{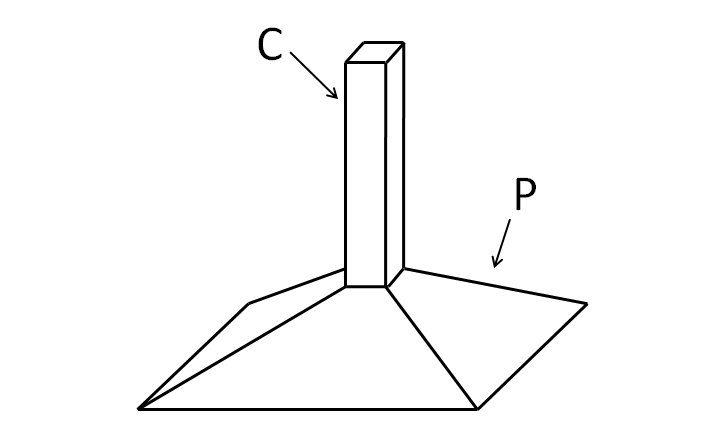}
\end{array}
$$

Note that each square $Q$ is replaced by a piecewise affine surface that is in the $\ell(Q)$-neighborhood of $Q$. Using this and the fact that $Q$ has distance $\gg\ell(Q)$ from $S_j\setminus F$, it is easy to see that $S_{j+1}$ is homeomorphic to $S_j$. Consequently, each $S_j$ is homeomorphic to $S_0=[0,1]^2$.
Also, it is clear from the geometry that $S_j$ converges to a surface $S$ homeomorphic to $[0,1]^2$. Later we will describe a homeomorphism $h_j: [0,1]^2 \rightarrow S_j$ explicitly and show that the homeomorphisms $h_j$ converge to a homeomorphism $h: [0,1]^2\rightarrow S$ that satisfy part (2) of our proposition. But first let us show part (1).

\subsection{Finite $\mathcal{H}^{2}$ measure}
Let $c>1$ be arbitrary. First note that for each cube $Q$ and its column $C$ and pyramid $P$, $\mathcal{H}^{2}(\partial(C))< 6\ell(Q^*)\ell(Q)$, and that $\mathcal{H}^{2}(\partial P \setminus Q)-\H^2( Q )$ is small if $\alpha(Q)$ is small. Therefore, by choosing $\alpha(Q)$ and $\ell(Q^*)$ small enough (so small that, even $\sum 6\ell(Q^*)\ell(Q)$ is small where the summation is taken over all squares $Q$ of all steps), we can ensure that $\mathcal{H}^{2}(S_j)<c$ for any $j$. 

In order to ensure that not only the surfaces $S_j$, but also the limit surface $S$ has finite $\mathcal{H}^{2}$ measure, inductively we define a sequence of positive numbers $\delta_j$ tending to zero, and a sequence of open sets $G_j$ such that $S_j\subset G_j$ and $\mathcal{H}^{2}_{\delta_j}(G_j)<c$. We can also assume that $G_0\supset \mbox{cl}\,G_1\supset G_1\supset \mbox{cl}\,G_2\supset G_2\dots$. Indeed, if $S_j$, $\delta_j$, $G_j$ have been already defined for a given $j$, then in the construction of $S_{j+1}$ we can choose the squares $Q$ so small that their sidelength is less then their distance from $G_j^c$. Then $S_{j+1}\subset G_j$; and we can choose $G_{j+1}$, $\delta_{j+1}$ with $\mbox{cl}\,G_{j+1}\subset G_j$ and $\mathcal{H}^{2}_{\delta_{j+1}}(G_{j+1})<c$.

Then $S\subset\bigcap_j \mbox{cl}\,G_j=\bigcap_j G_j$,  $\mathcal{H}^{2}_{\delta_j}(S)\le \mathcal{H}^{2}_{\delta_j}(G_j)<c$ for any $j$, and consequently, $\mathcal{H}^{2}(S)\le c$.

\begin{remark}{\rm We used the assumption $n>1$ in this proof. Indeed, in dimension $n$, the columns have surface area bounded by $2(n+1)\ell(Q^*)^{n-1}\ell(Q)$, which we can make small by choosing $\ell(Q^*)$ small enough, only if $n>1$.

We will not use this assumption anywhere else. We can do the same construction also for $n=1$, it will give us a nowhere differentiable $S$ with
$\mathcal{H}^{1}(S)=\infty$. As we have seen, part (1) of Proposition \ref{propp} is false for $n=1$. But part (2) will remain true; we will use the construction \ref{S} also in dimension $n=1$ when we prove part (2).}
\end{remark}

\subsection{Edges and faces}
In our construction, each face of $S_j$ is divided into countably many squares, and then each square is replaced by a piecewise affine surface, to obtain $S_{j+1}$. Naturally, we define the faces of $S_{j+1}$ as the faces of these piecewise affine surfaces; and it has two type of edges: firstly, the edges of $S_j$ are also edges of $S_{j+1}$, and secondly, the edges of our countably many piecewise affine surfaces. The points of the first type of edges do not belong to the closure of any face of $S_{j+1}$; while, at the points of the second type, at least two faces of $S_{j+1}$ meet, in a positive angle.

Note that all edges of $S_j$ are preserved by our construction: they are also edges of the surfaces $S_{j+1}, S_{j+2},\dots$, and consequently, they belong to the surface $S$.

\subsection{No tangents}
Next, we show that $S$ does not have a tangent plane at any of its points. First, consider an edge of $S_j$ for some $j\ge 1$, and assume that it is not on an edge of $S_{j-1}$. Then two faces of $S_j$ meet at this edge in a positive angle, so clearly $S_j$ does not have a tangent at the points of this edge. This non existence is preserved during our construction because, on both faces, there is a sequence of squares of $S_{j+1}$ converging to this edge, and the edges of these squares in $S_{j+1}$ will be preserved for every step afterwards. That is, for any point $x$ in any edge, ${\mbox contg}_S(x)$ contains all faces that meet at $x$.  

Now consider an $x\in S$ that does not belong to any edge of any $S_j$. Then, there is an arbitrary small $r$, such that within the $r$-neighborhood of $x$, there is a $Q\subset S_j$ of length comparable to $r$. This follows from our choice of the side lengths of the squares $Q$: if $Q$ is a square of our construction, 
then the squares $Q'$ used in the next step on the surface of the pyramid $P=P(Q)$ and the column $C=C(Q)$ all have side lengths $\ell(Q')\ll\ell(Q)$. And in the next steps $\ell(Q'')\ll\ell(Q')$, etc, therefore during the whole construction, we stay within a neighborhood $\ll \ell(Q)$. This also shows that on an arbitrary small scale $r$, next to $x$ in a distance comparable to $r$, $S$ ``looks like'' Figure 1, and not like a plane: it is not within a small angle of any plane. This finishes the proof of part (1).

\subsection{Definition of $h$}
Let us now define a homeomorphism $[0,1]^n\rightarrow S\subset\R^{n+1}$ which will satisfy the requirements in part (2). We will use the construction of the surface $S$ and its properties described in 4.1 and 4.3-4.4.

First, for each $j$, we define a homeomorphism from $S_0=[0,1]^n\subset[0,1]^{n+1}$ onto $S_j$.  Let $h_0$ be the identity. We define $h_j$ as $f_j \circ h_{j-1}$ where $f_j:S_{j-1}\rightarrow S_j$ is a homeomorphism.
We let $f_j$ to be the identity on $S_{j-1}\cap S_j$, and construct $f_j$ on each cube $Q \subset S_{j-1}$ individually. Let $\theta$ denote the direction orthogonal to $Q$.
\begin{enumerate}
\item For each $y \in Q\setminus Q^*$, let $f_j$ be the projection of $y$ onto $\partial P$ in direction $\theta$.
\item The cube $Q^*$ is first mapped by a translation in the direction $\theta$ onto the cube $Q'$ defined by $\partial Q'=\partial C\cap\partial P$, and then the interior of this cube $Q'$ is mapped to $\partial C\setminus\partial P$ by a perspective projection from a point far below the center of $Q^*$ in direction $\theta$. 
\end{enumerate}
$$
\begin{array}{c} \mbox{Figure 2}\\
\includegraphics[scale=0.4]{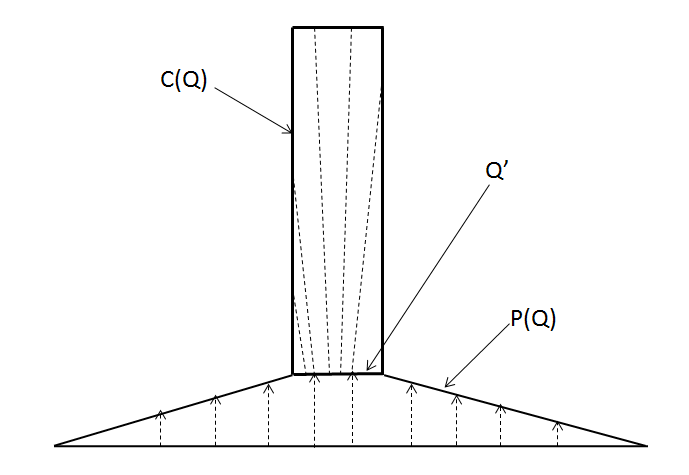}
\end{array}
$$
 We note that the point used for the perspective projection can be defined far enough from $Q^*$ so that the projection increases distances between points.  Since the $S_j$ are polygonal, $f_j$ is piecewise smooth as well as $h_j:=f_j \circ h_{j-1}$ for $j\ge 1$.  

From section 4.1, we know that $f_j(Q)$ lies within an $\ell(Q)$ neighborhood of $Q$. Therefore $c_j:=\|h_j-h_{j-1}\|_\infty\lesssim\sup l(Q)$, where the supremum is taken over all cubes of $S_{j-1}$. In 4.4, we have also seen that $c_j$ tends to zero (at least) as fast as a geometric sequence, therefore the continuous functions $h_j$ converge uniformly to a continuous function $h$. We need to show that $h$ is injective. 

First note that, for any face $F$ of $S_{j-1}$ and for any two $x,y\in F$, $\|f_j(x)-f_j(y)\|\gtrsim\|x-y\|$. Now let $x,y\in S_0$ be arbitrary, and consider the first $j$ for which $h_{j-1}(x),h_{j-1}(y)$ are not inside the same face of $S_{j-1}$. If $x$ is on an edge of $S_{j-1}$ then $h(x)=h_{j-1}(x)$, and if it is covered by one of our cubes $Q$, then $\|h_{j-1}(x)-h_j(x)\|\lesssim \ell(Q)$ and $\|h_j(x)-h(x)\|\ll \ell(Q)$. We also know that $Q$ has distance $\gg\ell(Q)$ from $S_{j-1}\setminus F$, where $F$ is the face that $Q$ belongs to. In particular, it has distance $\gg\ell(Q)$ from $y$. The same estimates are true for $y$. Putting these estimates together, we can see not only that $h$ is injective, but:
$$\|x-y\|\lesssim\|h(x)-h(y)\|.$$

This also shows that $h$ is not differentiable at any point $x$. Indeed, recall that $S$ cannot be well-approximated by a plane at any point, 
$$\|\theta\cdot (h(x)-h(y))\|=o(\|h(x)-h(y)\|)$$ fails for any direction $\theta$, and of course if $h$ is differentiable at $x$ with derivative $L:\mathbb{R}^n \rightarrow \mathbb{R}^{n+1}$, for $\theta \perp \mbox{Im }L$ we would get
$|\theta\cdot (h(x)-h(y))|=o(\|x-y\|)\lesssim  o(\|h(x)-h(y)\|)$.

\subsection{Lower scaled oscillation}
Consider a cube $Q$ on $S_{j}$, its middle cube $Q^*$, and the inverse image
$Q^{**}:=h_{j}^{-1}(Q^*)$. Let $E\subset S_0$ denote the set of points that belong to a $Q^{**}$ for infinitely many $j$. Since $\diam(Q^{**})\le c\,\diam(Q^*)$ for some absolute constant $c$,
%\ll\ell(Q)\lesssim\diam(f_j(Q^*))\lesssim\diam(h(Q^{**})),$$ therefore $L_h(x)=\infty$ at every $x\in E$ and thus $h$ is not differentiable at any point of $E$. 
%
%The inequality $\diam(Q^{**})\le\diam(Q^*)$ also shows that, 
by choosing the cubes $Q^*$ in our construction small enough (so small that, even $\sum g(c\,\diam(Q^*))<\infty$ where $g$ is the gauge function and the summation is taken over all cubes $Q^*$ in all steps), we indeed get $\H^g(E)=0$.

Now consider an $x$ in the complement of $E$. Then there is a $j$ such that $h_{k}(x)$ does not belongs to a column in $S_{k}$ for any $k\ge j$. Assume first that $x$ is not mapped to an edge of any $S_k$. 
%We postpone the discussion of what happens at the edges to 4.9.

In this case we know that $h_j$ is smooth around $x$, in particular, it is Lipschitz with some Lipschitz constant $L$ (that depends on $x$ and $j$). For $k\ge j$, let $Q_k$ denote the cube in $S_k$ that $h_{k}(x)$ belongs to. Then $h_k\circ h_j^{-1}$ is Lipschitz with
Lipschitz constant $\alpha:=1+\sum\alpha(Q)$ on $h_j\circ h_k^{-1}(Q_k)$, where $\alpha(Q)$ is the angle of the pyramid $P(Q)$ as in 4.1, and the summation is taken over all cubes of all steps so that $\alpha$ does not depend on $k$.

Recall that for any other cube $Q_k'$ in $S_k$, if $3Q_k\cap 3Q_k'\neq\emptyset$ then they are of comparable sizes. Therefore, all the cubes that meet $B_k:= B(h_k(x),\ell(Q_k))$ in $S_k$ are comparable to $Q_k$, therefore the oscillation of $h\circ h_k^{-1}$ on $B_k$ is comparable to $\ell(Q_k)$, in particular, it is bounded by $C\ell(Q_k)$ for some absolute constant $C$.

Now we are ready to prove that $l_h(x)<\infty$. For any $k$, consider the ball
$B(x,r)$ for a small $r$. This ball is mapped into the ball $B(h_j(x),Lr)$ by $h_j$,
then into $B(h_k(x),\alpha Lr)$ by $h_k$. Therefore, with the choice $r:=\ell(Q_k)/\alpha L$, the image of $B(x,r)$ through $h_k$ is inside $B_k$, and then the image of $h$ is inside the
$C\ell(Q_k)$-neighborhood of $B_k$. That is, 
$$\sup_{B(x,\ell(Q_k)/\alpha L)}\|h(x)-h(y)\|\lesssim\ell(Q_k).$$
Letting $k\to\infty$ we obtain $l_h(x)<\infty$.

\medskip
Finally, when $h_k(x)$ is on an edge of $S_k$ for some $k$, suppose that $k$ is the first index when this happens. Then it is clear that $l_{h_k}(x)<\infty$, and as we proceed with the construction, we only modify our functions on cubes $Q$ whose distance from $h_k(x)$ is $\gg \ell(Q)$, therefore in a neighborhood of $x$, $|h(y)-h(x)|$ is comparable to $|h_k(y)-h_k(x)|$.

\bibliographystyle{amsplain}

\end{document}